\documentclass[11pt,reqno]{amsart}
\usepackage{amssymb}
\usepackage{mathtools}
\usepackage{mathabx}
\usepackage{mathrsfs}
\usepackage{verbatim}  
\usepackage{wrapfig}
\usepackage{xcolor}
\usepackage{hyperref}
\hypersetup{
	colorlinks,
	linkcolor={red!50!black},
	citecolor={blue!50!black},
	urlcolor={blue!80!black}
}

\newtheorem{thm}{Theorem}[section]
\newtheorem{lem}[thm]{Lemma}

\newtheorem{cor}[thm]{Corollary}

\newtheorem{prop}[thm]{Proposition}

\theoremstyle{remark}
\newtheorem{rem}[thm]{Remark}
\newtheorem{Q}[thm]{Question}

\theoremstyle{definition}
\newtheorem{Def}[thm]{Definition}

\newcommand{\ra}{\rightarrow}
\newcommand{\Ra}{\Rightarrow}

\newcommand{\es}{\emptyset}
\newcommand{\ci}{\subseteq}

\newcommand{\A}{\mathcal{A}}
\newcommand{\al}{\alpha}
\newcommand{\be}{\beta}
\newcommand{\de}{\delta}
\newcommand{\e}{\varepsilon}

\newcommand{\la}{\lambda}

\newcommand{\ga}{\gamma}

\newcommand{\mb}{\mathbb}
\newcommand{\mc}{\mathcal}

\newcommand{\iy}{\infty}
\newcommand{\msc}{\mathscr}

\newcommand{\trcp}{\textrm{r.c.p.}}

\newcommand{\uc}{u.H.s.c.}
\newcommand{\lc}{l.H.s.c.}
\newcommand{\beqa}{\begin{eqnarray*}}
	\newcommand{\eeqa}{\end{eqnarray*}}

\newcounter{cnt1}
\newcounter{cnt2}
\newcounter{cnt3}
\newcounter{cnt4}
\newcommand{\blr}{\begin{list}{$($\roman{cnt1}$)$} {\usecounter{cnt1}
			\setlength{\topsep}{0pt} \setlength{\itemsep}{0pt}}}
	\newcommand{\blR}{\begin{list}{\Roman{cnt4}.\ } {\usecounter{cnt4}
				\setlength{\topsep}{0pt} \setlength{\itemsep}{0pt}}}
		\newcommand{\bla}{\begin{list}{$(\alph{cnt2})$} {\usecounter{cnt2}
					\setlength{\topsep}{0pt} \setlength{\itemsep}{0pt}}}
			\newcommand{\bln}{\begin{list}{$($\arabic{cnt3}$)$} {\usecounter{cnt3}
						\setlength{\topsep}{0pt} \setlength{\itemsep}{0pt}}}
				\newcommand{\el}{\end{list}}

\begin{document}
					\title[On property-$(P_1)$ and semi-continuity properties]{On property-$(P_1)$ and semi-continuity properties of restricted Chebyshev-center maps in $\ell_{\iy}$-direct sums}
				
				\author[T. Thomas]{Teena Thomas}
				\address[Teena Thomas]{Indian Institute of Technology Hyderabad\\ Sangareddy\\Telangana, India-502284}
				\email{ma19resch11003@iith.ac.in}
				
				\begin{abstract}
					For a compact Hausdorff space $S$, we prove that the closed unit ball of a closed linear subalgebra of the space of real-valued continuous functions on $S$, denoted by $C(S)$, satisfies property-$(P_1)$ (the set-valued generalization of strong proximinality) for the non-empty closed bounded subsets of the bidual of $C(S)$. Various stability results related to property-$(P_1)$ and semi-continuity properties of restricted Chebyshev-center maps are also established. As a consequence, we derive that if $Y$ is a proximinal finite co-dimensional subspace of $c_0$ then the closed unit ball of $Y$ satisfies property-$(P_1)$ for the non-empty closed bounded subsets of $\ell_{\iy}$ and the restricted Chebyshev-center map of the closed unit ball of $Y$ is Hausdorff metric continuous on the class of non-empty closed bounded subsets of $\ell_{\iy}$. We also investigate a variant of the transitivity property, similar to the one discussed in [C. R. Jayanarayanan and T. Paul, {\em Strong proximinality and intersection properties of balls in Banach spaces}, J. Math. Anal. Appl., 426(2):~1217--1231, 2015], for property-$(P_1)$.
				\end{abstract}
				
				\subjclass[2010]{Primary 41A65; Secondary 46E15, 46A55, 46B20}
				
				\keywords{property-$(P_1)$, subalgebra, polyhedral Banach space, restricted Chebyshev-center map, lower and upper Hausdorff semi-continuity}
				\maketitle
				
				\section{Introduction}
				In this article, we explore the (semi-)continuity properties of the restricted Chebyshev-center maps in certain classes of Banach spaces through the notion of property-$(P_1)$, which is the set-valued analogue of strong proximinality. For a real Banach space $X$, a non-empty closed convex subset $V$ of $X$ and a family $\msc{F}$ of non-empty closed bounded subsets of $X$, property-$(P_1)$ is defined for a triplet $(X,V,\msc{F})$. The main motivation to study this property is 
				\begin{thm}[{\cite[Theorem~5]{Mach2}}]\label{Rem2}
					If the triplet $(X,V,\msc{F})$ has property-$(P_1)$ then the restricted Chebyshev-center map of $V$, which maps every set $F \in \msc{F}$ to the set of restricted Chebyshev centers of $F$ in $V$, is upper semi-continuous on $\msc{F}$ with respect to the Hausdorff metric.
				\end{thm}	
			
				Throughout this paper, $X$ denotes a real Banach space. We assume only norm closed and linear subspaces. For $x \in X$ and $r>0$, let $B_{X}[x,r] = \{y \in X\colon \|x-y\| \leq r\}$ and $B_{X}(x,r)= \{y \in X\colon \|x-y\| < r\}$. For simplicity, we denote $B_{X}[0,1]$ by $B_X$. The closed unit ball of a subspace $Y$ of $X$ is the set $B_{Y} \coloneqq B_X \cap Y$. Let $\mc{CV}(X)$ be the class of all non-empty closed convex subsets of $X$. For $V \in \mc{CV}(X)$, we denote the classes of all non-empty closed bounded, compact and finite subsets of $V$ by $\mc{CB}(V)$, $\mc{K}(V)$ and $\mc{F}(V)$ respectively. For $x \in X$ and a bounded subset $A$ of $X$, we define $d(x,A) = \inf \{\|x-a\|\colon a \in A\}$ and $r (x,A) = \sup \{\|x-a\|\colon a \in A\}$. For a set $A \ci X$ and $\la>0$, let $\la A = \{\la a \colon a \in A\}$. 
				
				Let $V \in \mc{CV}(X)$ and $B \in \mc{CB}(X)$. The \textit{restricted Chebyshev radius} of $B$ in $V$ is defined as the number $\emph{rad}_{V}(B) \coloneqq \inf_{v \in V} r(v,B)$. The elements of the set $\emph{cent}_{V}(B) \coloneqq \{v \in V\colon \emph{rad}_{V}(B) = r(v,B)\}$ are called the \textit{restricted  Chebyshev centers} of $B$ in $V$. For each $\de>0$, we define $\emph{cent}_{V}(B,\de) = \{v \in V\colon r(v,F) \leq \emph{rad}_{V}(F) + \de\}$. 
				
				Let $\msc{F} \ci \mc{CB}(X)$. The pair $(V,\msc{F})$ satisfies the \textit{restricted center property} $($\trcp$)$ if for each $F \in \msc{F}$, $\emph{cent}_{V}(F) \neq \es$. The set-valued map, $\emph{cent}_{V}(.)$, which maps each set $F \in \msc{F}$ to the set $\emph{cent}_{V}(F)$, is called the \textit{restricted Chebyshev-center map} of $V$ on $\msc{F}$. We refer to \cite{PN} for the terminologies defined above.
				
				In particular, let $\msc{F}$ be the class of all singleton subsets of $X$. If $(V, \msc{F})$ has \trcp~then $V$ is said to be \textit{proximinal} in $X$ and the restricted Chebyshev-center map of $V$ on $\msc{F}$, denoted by $P_{V}$, is called the \textit{the metric projection} from $X$ onto $V$.
				
				We discuss the continuity properties of the metric projections and restricted Chebyshev-center maps with respect to the Hausdorff metric. We recall that for a Banach space $X$, the Hausdorff metric, denoted by $d_H$, is defined as follows$\colon$ for each $A,B \in \mc{CB}(X)$, \[d_{H}(A,B) = \inf\{t>0\colon A \ci B + t B_{X}\mbox{ and }B \ci A + t B_{X}\}.\]
				
				Let $T$ be a topological space. Let $\Phi$ be a set-valued map from $T$ into $\mc{CB}(X)$. We say $\Phi$ is \textit{lower Hausdorff semi-continuous} (\lc) on $T$ if for each $t \in T$ and $\e>0$, there exists a neighbourhood $\msc{N}(t,\e)$ of $t$ such that for each $s \in \msc{N}(t,\e)$ and $z \in \Phi(t)$, $B_{X}(z,\e) \cap \Phi(s) \neq \es.$ The map $\Phi$ is \textit{upper Hausdorff semi-continuous} (\uc) on $T$ if for each $t \in T$ and $\e>0$, there exists a neighbourhood $\msc{N}(t,\e)$ of $t$ such that for each $s \in \msc{N}(t,\e)$, $\Phi(s) \ci \Phi(t) + \e B_X.$ The map $\Phi$ is \textit{Hausdorff metric continuous} on $T$ if for each $t \in T$, the single-valued map $\Phi$ from $T$ into the metric space $(\mc{CB}(X), d_{H})$ is continuous at $t \in T$. By \cite[Remark~2.8]{VI}, $\Phi$ is Hausdorff metric continuous on $T$ if and only if $\Phi$ is both \lc~and \uc~on $T$.
				
				In this article, we consider the following definition of property-$(P_1)$ which is equivalent to that introduced by J. Mach in \cite{Mach2}.
				\begin{Def}[{\cite[Definition~1.1]{LPST}}]\label{Def1}
					Let $X$ be a Banach space, $V \in \mc{CV}(X)$ and $\msc{F} \ci \mc{CB}(X)$ such that $(V,\msc{F})$ has \trcp. Then the triplet $(X,V,\msc{F})$ has \textit{property-$(P_1)$} if for each $\e>0$ and $F \in \msc{F}$, there exists $\de(\e,F)>0$ such that \[\emph{cent}_{V}(F,\de) \ci \emph{cent}_{V}(F) + \e B_{X}.\]
				\end{Def}
				
				\begin{rem}\label{rem2}
					With the notations as above, let $(V,\msc{F})$ have \trcp.
					\blr
					\item If $\msc{F}$ is the class of all singleton subsets of $X$ then we say $V$ is \textit{strongly proximinal} in $X$ (see \cite{GI}). 
					\item It is easy to verify that $(X,V,\msc{F})$ has property-$(P_1)$ if and only if for each $\epsilon>0$ and $F \in \msc{F}$, there exists $\de>0$ such that  \[\mc{S}(F,\de) \coloneqq \sup \{d(v,\emph{cent}_{V}(F))\colon v \in V\mbox{ and }r(v,F) \leq \emph{rad}_{V}(F) +\de\} < \e.\]
					\el
				\end{rem}
				
				We now recall a notion in Banach spaces which is stronger than strong proximinality. A subspace $J$ of a Banach space $X$ is an \textit{$M$-ideal} in $X$ if there exists a linear projection $P\colon X^\ast \ra X^\ast$ such that for each $x^\ast \in X^\ast$, $\|x^\ast\| = \|Px^\ast\| + \|x^\ast - Px^\ast\|$ and the range of $P$ is the annihilator of $J$ in $X^\ast$. A particular case of $M$-ideals are $M$-summands. A subspace $J$ of a Banach space $X$ is an \textit{$M$-summand} in $X$ if there exists a linear projection $P\colon X \ra X$ such that for each $x \in X$, $\|x\| = \max \{\|Px\|, \|x - Px\|\}$ and the range of $P$ is $J$. We refer to \cite{HWW} for a detailed study on $M$-ideals. $M$-ideals are strongly proximinal; see \cite{HWW} and \cite[Proposition~3.3]{DuNa2}. However, the example provided by Vesel\'{y} in \cite{Ves} shows that in general, an $M$-ideal in a Banach space may not admit Chebyshev centers for its closed bounded subsets. 
				
				For a topological space $T$ and a Banach space $X$, $C_{b}(T,X)$ denote the Banach space of $X$-valued bounded continuous functions on $T$. If $T$ is a compact Hausdorff space and $X = \mb{R}$ then $C_{b}(T,X)$ is simply denoted as $C(T)$. 
				
				In \cite{GV}, Godefroy and Indumathi proved that for a subspace $X$ of $c_0$ and a finite co-dimensional linear subspace $Y$ of $X$, if every hyperplane of $X$ containing $Y$ is proximinal in $X$ then $Y$ is proximinal in $X$.  
				Later, Indumathi proved the following result in \cite[Theorem~4.1]{VI}.
				\begin{thm}\label{Th1}
					Let $Y$ be a proximinal finite co-dimensional linear subspace of $c_0$. Then $Y$ is proximinal in $\ell_{\iy}$ and the metric projection $P_Y$ from $\ell_{\iy}$ onto $Y$ is Hausdorff metric continuous. 
				\end{thm}  
				Jayanarayanan and Lalithambigai improved Theorem~\ref{Th1} in \cite[Corollary~3.7]{CRLa}. They proved that if $Y$ is a proximinal finite co-dimensional subspace of $c_0$ then $B_Y $ is strongly proximinal in $\ell_{\iy}$.
				
				The present paper generalizes Theorem~\ref{Th1} and proves that the triplet $(\ell_{\iy}, B_Y,\mc{CB}(\ell_{\iy}))$ satisfies property-$(P_1)$ and the map $\emph{cent}_{B_Y}(.)$ Hausdorff metric continuous on $\mc{CB}(\ell_{\iy})$ whenever $Y$ is a proximinal finite co-dimensional subspace of $c_0$ in Sections~\ref{S3} and \ref{S4} respectively. To this end, we recall 
				\begin{thm}[{\cite[Theorem~2]{Amir}}]\label{T8}
					Let $T$ be a topological space and $X$ be a uniformly convex Banach space. Then for each $B \in \mc{CB}(C_{b}(T,X))$, $\textrm{cent}_{C_{b}(T,X)}(B) \neq \es$ and the map $\textrm{cent}_{C_{b}(T,X)}(.)$ is uniformly Hausdorff continuous on $\mc{CB}(C_{b}(T,X))$.
				\end{thm}
			
				In Section~\ref{S2}, we improve Theorem~\ref{T8} and obtain a few consequences which are of interest. We prove that for a compact Hausdorff space $S$ and a uniformly convex Banach space $X$, $(C_{b}(T,X), B_{C_{b}(T,X)}, \mc{CB}(C_{b}(T,X)))$ has property-$(P_1)$ and the map $\emph{cent}_{B_{C_{b}(T,X)}}(.)$ is uniformly Hausdorff metric continuous on $\mc{CB}(C_{b}(T,X))$. Consequently, we obtain that $(\ell_{\iy}, B_{c_0}, \mc{CB}(\ell_{\iy}))$ has property-$(P_1)$ which generalizes \cite[Theorem~3.1]{CRLa}. Moreover, we prove that for a compact Hausdorff space $S$ and a closed linear subalgebra $\A$ of $C(S)$, the triplet $(C(S)^{\ast \ast},B_{\A}, \mc{CB}(C(S)^{\ast \ast}))$ satisfies property-$(P_1)$. 
				
				In Section~\ref{S3}, apart from the results mentioned earlier, we also establish the stability of \trcp~and property-$(P_1)$ in the $\ell_{\iy}$-direct sums. Given two subspaces $Y_1$ and $Y_2$ of Banach spaces $X_1$ and $X_2$, let $X$ be the $\ell_{\iy}$-direct sum of $X_1$ and $X_2$ and $Y$ be that of $Y_1$ and $Y_2$. We prove that $(X,Y,\mc{CB}(X))$ has property-$(P_1)$ if and only if for each $i=1,2$, $(X_i,Y_i,\mc{CB}(X_i))$ has property-$(P_1)$. 
				
				In Section~$\ref{S4}$, we derive the stability of the lower and upper semi-continuity properties of the restricted Chebyshev-center maps of $\ell_{\iy}$-direct sums under certain assumptions. It is proved in \cite{T1} that for a subspace $Y$ of a Banach space $X$, if $(X,B_Y,\mc{CB}(X))$ has property-$(P_1)$ then so does $(X,Y,\mc{CB}(X))$. In this section, we establish that for a subspace $Y$ of a Banach space $X$, if $(B_Y,\mc{CB}(X))$ has \trcp~then the Hausdorff metric continuity of the map $\emph{cent}_{B_Y}(.)$ on $\mc{CB}(X)$ implies that of the map $\emph{cent}_{Y}(.)$ on $\mc{CB}(X)$. 
				
				 In \cite{CT}, the authors proved a variation of the transitivity problem for strong proximinality. 
				\begin{thm}[{\cite[Theorem~3.13]{CT}}]\label{T9}
					Let $X$ be a Banach space, $J$ be an $M$-ideal in $X$ and $Y$ be a finite co-dimensional subspace of $X$ such that $Y \ci J \ci X$. If $Y$ is strongly proximinal in $J$, then $Y$ is strongly proximinal in $X$.
				\end{thm}
				
				We pose a similar problem for property-$(P_1)$ as follows$\colon$
				
				\begin{Q}\label{Q2}
					Let $X$ be a Banach space, $J$ be an $M$-ideal in $X$ and $Y$ be a finite co-dimensional subspace of $X$ such that $Y \ci J$. If the triplets $(J,Y,\mc{CB}(J))$ (or $(J,Y,\mc{K}(J))$ or $(J,Y,\mc{F}(J))$ respectively) and $(X,J,\mc{CB}(X))$ (or $(X,J,\mc{K}(X))$ or $(X,J,\mc{F}(X))$ respectively) have property-$(P_1)$, then does $(X,Y,\mc{CB}(X))$ (or $(X,Y,\mc{K}(X))$ or $(X,Y,\mc{F}(X))$ respectively) have property-$(P_1)$? 
				\end{Q}	   
				We do not know the answer to Question~\ref{Q2} in its entirety. Nevertheless, we provide positive results for the following two particular cases in Section~\ref{sec5}$\colon$ We prove that for a Banach space $X$, an $M$-summand $Y$ in $X$ and a subspace $Z$ of $Y$, if $(Y,Z,\mc{CB}(Y))$ has property-$(P_1)$ then so does $(X,Z,\mc{CB}(X))$. Further, let us recall that a Banach space $X$ is said to be an \textit{$L_1$-predual space} if $X^{\ast}$ is isometric to an $L_1(\mu)$ space for some positive measure $\mu$. We refer to \cite{LL} for a detailed study on $L_1$-preduals. We prove that for an $L_1$-predual space $X$, an $M$-ideal $J$ in $X$ and a finite co-dimensional subspace $Y$ of $X$ such that $Y \ci J \ci X$, if $Y$ is strongly proximinal in $J$, then $(X,Y,\mc{K}(X))$ has property-$(P_1)$.
				
				\section{Property-$(P_1)$ in vector-valued continuous function spaces}\label{S2}
				
				In this section, we discuss property-$(P_1)$ of the closed unit ball of the space $C_{b}(T,X)$ whenever $T$ is a topological space and $X$ is a uniformly convex Banach space.
				To this end, we first recall the following characterization of a uniformly convex Banach space. 
				
				\begin{lem}[{\cite[Lemma~1]{Amir}}]\label{L1}
					Let $X$ be a Banach space. Then $X$ is uniformly convex if and only if for each $\e>0$, there exists $\de^{\prime}(\e)>0$ such that if $x,y \in X$ and $\Phi \in X^{\ast}$ such that $\|x\|=\|y\| = 1 = \|\Phi\| = \Phi(y)$ and $\Phi(x)> 1- \de^{\prime}(\e)$ then $\|x-y\|<\e$. We can choose $\de^{\prime}(\e) \leq \frac{\e}{2}$. 
				\end{lem}
				
				We also recall the following estimation. 
				\begin{lem}[{\cite[Lemma~4.1]{T1}}]\label{L2}
					Let $X$ be a Banach space and $V \in \mc{CV}(X)$. Then for each $A,B \in \mc{CB}(X)$ and $v \in V$, $\vert r(v,A) - r(v,B)\vert\leq d_{H}(A,B)$ and $\vert \textrm{rad}_{V}(A) - \textrm{rad}_{V}(B) \vert \leq d_{H}(A,B)$.
				\end{lem}
				
				The following result is obtained through some modifications in the proof of \cite[Theorem~2]{Amir}. For the sake of thoroughness, we present the modifications in its proof here.
				\begin{thm}\label{P9}
					Let $T$ be a topological space and $X$ be a uniformly convex Banach space. Then the triplet  $(C_{b}(T,X), B_{C_{b}(T,X)}, \mc{CB}(C_{b}(T,X)))$ satisfies property-$(P_1)$ and the map $\textrm{cent}_{B_{C_{b}(T,X)}}(.)$ is uniformly Hausdorff metric continuous on $\mc{CB}(C_{b}(T,X))$.
				\end{thm}
				\begin{proof}
					Let $B \in \mc{CB}(C_{b}(T,X))$. We define $R = \emph{rad}_{B_{C_{b}(T,X)}}(B)$. Let us assume $R =1$ and fix $\e>0$. Then we obtain $\de^{\prime}(\e)>0$ satisfying the condition in Lemma~\ref{L1}. Now there exists $f_0 \in B_{C_{b}(T,X)}$ such that $r(f_0,B) \leq 1+ \de^{\prime}(\e)$. We claim that there exists $f_1 \in B_{C_{b}(T,X)}$ such that $r(f_1,B) \leq 1+ \de^\prime (\e/2)$ and $\|f_0 - f_1\| \leq 2 \e$. 
					Indeed, there exists $g \in B_{C_{b}(T,X)}$ such that $r(g,B) \leq 1+ \de^\prime (\e/2)$. We now define $\al \colon T \ra [0,1]$ and $f_1 \colon T \ra X$ as follows$\colon$ for each $t \in T$,
					\begin{equation*}
						\al(t)=
						\begin{dcases}
							1, & \mbox{if }\|g(t) - f_0(t)\| \leq 2\e; \\
							\frac{2\e}{\|g(t) - f_0(t)\|}, & \mbox{if }\|g(t) - f_0(t)\| > 2 \e. 
						\end{dcases}
					\end{equation*}
					and  \[f_1 (t)=  f_0(t) + \al(t) ( g(t) - f_0(t)).\] Clearly, $f_1 \in B_{C_{b}(T,X)}$ and $\|f_1 - f_0\| \leq 2 \e$. For each $b \in B$, we apply the same arguments as in the proof of \cite[Theorem~2]{Amir} to show that for each $t \in T$, 
					\begin{equation}\label{Eq1}
						\|f_1(t) - b(t)\| \leq 1 + \de^\prime (\e/2)
					\end{equation} 
					and hence $r(f_1,B) \leq 1 + \de^{\prime}(\e/2)$.
					
					We now proceed inductively to find a sequence $\{f_n\} \ci B_{C_{b}(T,X)}$ such that for each $n=1,2,\ldots$, $\|f_{n+1}-f_n\| \leq 2 \frac{\e}{2^{n}}$ and $r(f_{n+1}, B) \leq 1+ \de^\prime ( \e/2^{n+1})$. Since $\{f_n\}$ is Cauchy, there exists $f \in B_{C_{b}(T,X)}$ such that $\lim_{n \ra \iy} f_n = f$. Hence $\|f-f_0\| \leq 4\e$ and $r(f,B) \leq \lim_{n \ra \iy} r(f_n, B) \leq 1$. Thus $f \in \emph{cent}_{B_{C_{b}(T,X)}}(B)$.
					It also follows that $\emph{cent}_{B_{C_{b}(T,X)}}(B, \de^{\prime}(\e)) \ci\emph{cent}_{B_{C_{b}(T,X)}}(B) + 4\e B_{X}$. Hence $(C_{b}(T,X), B_{C_{b}(T,X)}, \{B\})$ has property-$(P_1)$.
					
					Now assume $0<R \neq 1$. Then $\inf_{f \in B_{C_{b}(T,X)}} r\left(\frac{f}{R}, \frac{1}{R}B\right)=1$. In the argument above, if $f_0, g$ are chosen in $\frac{1}{R} B_{C_{b}(T,X)}$ then $f_1 \in \frac{1}{R} B_{C_{b}(T,X)}$. Hence replacing $B_{C_{b}(T,X)}$ and $B$ by $\frac{1}{R} B_{C_{b}(T,X)}$ and $\frac{1}{R} B$ respectively in the argument above, we can conclude that $(C_{b}(T,X), \frac{1}{R} B_{C_{b}(T,X)}, \{\frac{1}{R} B\})$ has property-$(P_1)$. Thus it follows from \cite[Proposition~2.3~(ii)]{T1} that $(C_{b}(T,X), B_{C_{b}(T,X)}, \{B\})$ has property-$(P_1)$.
					
					In order to show that the map $\emph{cent}_{B_{C_{b}(T,X)}}(.)$ is uniformly Hausdorff continuous on $\mc{CB}(C_{b}(T,X))$, let us fix $\e, R>0$. We now obtain a $\de^{\prime}(\e)> 0$ satisfying the condition in Lemma~\ref{L1}. Choose $0 < \de< \frac{R  \de^{\prime}(\e)}{2}$. Let $A,B \in \mc{CB}(C_{b}(T,X))$ such that $\emph{rad}_{B_{C_{b}(T,X)}}(B), \emph{rad}_{B_{C_{b}(T,X)}}(A) < R$ and $d_{H}(B,A)< \de$. Then by Lemma~\ref{L2}, $\vert \emph{rad}_{B_{C_{b}(T,X)}}(B) - \emph{rad}_{B_{C_{b}(T,X)}}(A) \vert <\de.$ Let $f \in \emph{cent}_{B_{C_{b}(T,X)}}(A)$. Thus \[r(f,B) < \emph{rad}_{B_{C_{b}(T,X)}}(A) + \de < \emph{rad}_{B_{C_{b}(T,X)}}(B) + 2\de < (1 + \de^{\prime}(\e)) R.\] Now using the arguments above, we obtain $f_0 \in \emph{cent}_{B_{C_{b}(T,X)}}(B)$ such that $\|f-f_0\| \leq 4\e R$. Similarly, $\emph{cent}_{B_{C_{b}(T,X)}}(B) \ci \emph{cent}_{B_{C_{b}(T,X)}}(A) + 4\e R B_{C_{b}(T,X)}.$ Thus $d_{H}(\emph{cent}_{B_{C_{b}(T,X)}}(B)),\emph{cent}_{B_{C_{b}(T,X)}}(A)) \leq 4 \e R$.
				\end{proof}
				
				\begin{cor}\label{C2.1}
					Let $S$ be a compact Hausdorff space and $\A$ be a closed linear subspace of $C(S)$ described as follows$\colon$\[\A = \{f \in C(S)\colon f(t_i) = \la_i f(s_i)\mbox{, for each }i \in I\},\] for some index $I$ and co-ordinates $(t_i,s_i,\la_i) \in S \times S \times \{-1,0,1\}$ for each $i\in I$. Then the triplet $(C(S), B_\A, \mc{CB}(C(S)))$ satisfies property-$(P_1)$ and the map $\textrm{cent}_{B_{\A}}(.)$ is uniformly Hausdorff metric continuous on $\mc{CB}(C(S))$.
				\end{cor}
				\begin{proof}
					In the proof of Theorem~\ref{P9}, if we choose $f_0$ and $g$ in $B_\A$ then clearly $\|f_1\|\leq 1$ and from the description of $\A$, $f_1 \in B_\A$. Hence the result follows.
				\end{proof}
			By using the representation given in \cite[Theorem~2.2]{T3} of closed linear subalgebras of $C(S)$ and applying Corollary~\ref{C2.1}, we obtain the following result.
			 
				\begin{cor}\label{C2}
					Let $S$ be a compact Hausdorff space and $\A$ be a closed linear subalgebra of $C(S)$. Then the triplet $(C(S), B_\A, \mc{CB}(C(S)))$ satisfies property-$(P_1)$ and the map $\textrm{cent}_{B_{\A}}(.)$ is uniformly Hausdorff metric continuous on $\mc{CB}(C(S))$.
				\end{cor}
				
				For a compact Hausdorff space $S$, it is known that $C(S)^{\ast \ast}$ is a $C(\Omega)$ space, for some compact Hausdorff space $\Omega$ (see \cite{Lacey}) and $C(S)$ is a closed linear subalgebra of $C(\Omega)$ under the canonical embedding (see \cite{Seever}). Therefore, as a direct consequence of Corollary~\ref{C2}, we have
				
				\begin{cor}\label{C3}
					Let $S$ be a compact Hausdorff space and $\A$ be a closed linear subalgebra of $C(S)$. Then the triplet $(C(S)^{\ast \ast}, B_\A, \mc{CB}(C(S)^{\ast \ast}))$ satisfies property-$(P_1)$ and the map $\textrm{cent}_{B_{\A}}(.)$ is uniformly Hausdorff metric continuous on $\mc{CB}(C(S)^{\ast \ast})$.
				\end{cor}
				
				The triplet $(\ell_{\iy}, B_{c_0},\mc{K}(\ell_{\iy}))$ has property-$(P_1)$, and it follows from the well-known fact that $c_0$ is an $M$-ideal in $\ell_{\iy}$ and \cite[Theorem~3.2]{T1}. The next result follows from the fact that $c_0$ is a subalgebra in $\ell_{\iy} \cong C(\be \mathbb{N})$ (here $\be \mathbb{N}$ is the Stone-\v{C}ech compactification of the natural numbers $\mathbb{N}$) and Corollary~\ref{C2}. 
				\begin{cor}\label{C1}
					The triplet $(\ell_{\iy}, B_{c_0}, \mc{CB}(\ell_{\iy}))$ has property-$(P_1)$ and the map $\textrm{cent}_{B_{c_0}}(.)$ is uniformly Hausdorff metric continuous on $\mc{CB}(\ell_{\iy})$.
				\end{cor}
				
				\section{Stability of property-$(P_1)$ in $\ell_{\iy}$-direct sums}\label{S3}
				
				We first establish some notations which are used in the present and subsequent sections. Let $X_1$ and $X_2$ be two Banach spaces. Then the $\ell_{\iy}$-direct sum of $X_{1}$ and $X_{2}$, \[X\coloneqq X_{1} \oplus_{\iy} X_{2} = \{(x_1,x_2) \in X_1 \times X_2 \colon x_1 \in X_1\mbox{ and }x_2 \in X_2\},\] is a Banach space equipped with the maximum norm. For $i\in\{1,2\}$, if $V_i \in \mc{CV}(X_i)$ then $V_1 \times V_2 \in \mc{CV}(X)$.
				For each $B \in \mc{CB}(X)$, let
				\begin{equation}
					\begin{split}
						&B(1) = \{b_1 \in X_1\colon \mbox{there exists }b \in B\mbox{ and }b_2 \in X_2\mbox{ such that }b=(b_1,b_2)\}\\\mbox{and }&B(2) = \{b_2 \in X_2\colon \mbox{there exists }b \in B\mbox{ and }b_1 \in X_1\mbox{ such that }b=(b_1,b_2)\}.
					\end{split}
				\end{equation}
			   Clearly, for each $i \in \{1,2\}$, $B(i) \in \mc{CB}(X_{i})$ and $B \ci B(1) \times B(2)$. 
				Moreover, for each $B \in \mc{CB}(X)$ and $i\in \{1,2\}$, let $r_{i}(B) = \emph{rad}_{V_{i}}(B(i)).$
				
				The following result gives a formula for the restricted Chebyshev radius of a closed bounded subset of an $\ell_{\iy}$-direct sum.
				\begin{prop}\label{Prop2}
					For each $i \in \{1,2\}$, let $X_i$ be a Banach space and $V_{i} \in \mc{CV}(X_i)$. Let $X = X_{1} \oplus_{\iy} X_2$, $V = V_{1} \times V_2$ and $B \in \mc{CB}(X)$. Then
					$\textrm{rad}_{V}(B) = \max \{r_{1}(B), r_{2}(B)\}$ and for each $v =(v_1,v_2) \in V$, $r(v,B) = \max \{r(v_1,B(1)),r(v_2,B(2))\}$.
				\end{prop}
				\begin{proof}
					Let $v=(v_1,v_2) \in V$ and for $i\in \{1,2\}$, $b_i \in B(i)$. Then there exists $b,b^{\prime} \in B$ and $b^{\prime}_{i} \in X_i$ for $i\in \{1,2\}$ such that $b = (b_1, b^{\prime}_{2})$ and $b^{\prime} = (b^{\prime}_{1}, b_{2})$. Thus $\|b_1 - v_1\| \leq \|b-v\| \leq r(v,B)$ and $\|b_2 - v_2\| \leq \|b^{\prime} - v\| \leq  r(v,B).$ It follows that for $i\in \{1,2\}$, $r(v_i, B(i)) \leq r(v,B)$. It is now easy to conclude that $\max\{r_{1}(B), r_{2}(B)\} \leq \emph{rad}_{V}(B).$
					
					Conversely, for each $\e>0$ and $i \in \{1,2\}$, there exists $v_i \in V_{i}$ such that $r(v_{i},B(i)) < r_{i}(B) + \e$. Let $v=(v_1,v_2) \in V$ and $b=(b_1,b_2) \in B$. Then for $i 
					\in \{1,2\}$, $b_i \in B(i)$. Hence 
					\begin{equation*}
						\begin{split}
							\|v-b\| = \max \{\|v_1 - b_1\|,\|v_2 - b_2\|\}< \max \{r_{1}(B),r_{2}(B)\}  + \e.
						\end{split}
					\end{equation*}
					It follows that $\emph{rad}_{V}(B) \leq  \max \{r_{1}(B),r_{2}(B)\}  + \e.$ Since $\e>0$ is arbitrary, we obtain the desired equality.
					
					It follows from the arguments above that for each $v =(v_1,v_2) \in V$, $r(v,B) = \max \{r(v_1,B(1)), r(v_2,B(2))\}$.
				\end{proof}
				
				We now establish the stability of \trcp~under the $\ell_{\iy}$-direct sum.
				\begin{prop}\label{Prop1}\label{Prop3}
					For each $i \in \{1,2\}$, let $X_i$ be a Banach space and $V_{i} \in \mc{CV}(X_i)$. Let $X = X_{1} \oplus_{\iy} X_2$ and $V = V_{1} \times V_2$. If for each $i \in \{1,2\}$, $(V_i, \mc{CB}(X_{i}))$ has \trcp~then $(V,\mc{CB}(X))$ has \trcp.
				\end{prop}
				\begin{proof}
					Let $B \in \mc{CB}(X)$. By our assumption, for $i\in \{1,2\}$, let $v_{i} \in \emph{cent}_{V_{i}}(B(i))$. Then it follows from Proposition~\ref{Prop2} that $\emph{rad}_{V}(B) = r((v_1,v_2), B)$.
				\end{proof}

				\begin{rem}\label{rem1}
					For each $i\in \{1,2\}$, let $X_i$ be a Banach space and $V_{i} \in \mc{CV}(X_i)$. For each $i \in \{1,2\}$, suppose $(V_i,\mc{CB}(X_i))$ has \trcp. Let $X = X_1 \oplus_{\iy} X_2$ and $V = V_1 \times V_2$. It is easy to verify the following facts.
					\blr
					\item For each $B \in \mc{CB}(X)$, 
					\begin{equation*}
						\emph{cent}_{V}(B)=
						\begin{dcases}
							\emph{cent}_{V_1}(B(1)) \times \emph{cent}_{V_{2}}(B(2)), & \mbox{if }r_{1}(B) = r_2(B); \\
							\bigcap_{b_{1} \in B(1)} B_{X_1}[b_{1}, r_{2}(B)] \cap V_1 \times \emph{cent}_{V_{2}}(B(2)), & \mbox{if }r_{1}(B) < r_{2}(B);\\ \emph{cent}_{V_{1}}(B(1)) \times \bigcap_{b_{2} \in B(2)} B_{X_2}[b_{2}, r_{1}(B)] \cap V_2, & \mbox{if }r_{2}(B) < r_{1}(B). 
						\end{dcases}
					\end{equation*}
					We note that in each case above, \[\emph{cent}_{V}(B) \supseteq \emph{cent}_{V_{1}}(B(1)) \times \emph{cent}_{V_{2}}(B(2)).\]
					\item For each $B,A \in \mc{CB}(X)$, \[\max \{d_{H}(B(1),A(1)),d_{H}(B(2),A(2))\} \leq d_{H}(B, A).\]
					\el
				\end{rem}
				
				Next, we prove a technical result. It is crucial to this study and is a generalization of \cite[Fact~3.2]{VI}.
				\begin{lem}\label{L4.6.1}
					Let $X$ be a Banach space, $V \in \mc{CV}(X)$ and $\msc{F} \ci \mc{CB}(X)$ such that $(V,\msc{F})$ has \trcp. Let $F \in \msc{F}$ and $\al > \textrm{rad}_{V}(F)$. Then for each $\e>0$, there exists $\de>0$ such that for each $F^\prime \in \msc{F}$ with $d_{H}(F,F^\prime) <\de$ and scalar $\be$ such that $\vert \al - \be \vert < \de$, we have \[d_{H}\left(\bigcap_{z \in F} B[z,\al] \cap V, \bigcap_{z^\prime \in F^\prime}B[z^\prime, \be] \cap V\right) <\e.\]
				\end{lem}
				\begin{proof}
					Let us define $R_F = \emph{rad}_{V}(F)$ and fix $\e>0$. Further, define $2 \ga = \al - R_F$, $L = \al + R_F  +2$ and $\de = \min \{1, \frac{\ga}{2}, \frac{\ga \e}{2L}\}$.
					
					Let $F^\prime \in \msc{F}$ be such that $d_{H}(F,F^\prime)<\de$ and $\be$ be a scalar such that $\vert \al -\be \vert < \de$. For simplicity, let $R_{F^{\prime}} = \emph{rad}_{V}(F^\prime)$. Then from Lemma~\ref{L2}, $\vert R_{F} - R_{F^\prime} \vert < \de$. Moreover, 
					\[\be - R_{F^{\prime}} = \be- \al  + R_F - R_{F^{\prime}} + \al - R_{F} > 2 \ga - 2 \de \geq 2 \ga - \ga = \ga.\]
					
					Now, let $v \in \bigcap_{z \in F} B[z,\al] \cap V$ and $z^\prime \in F^\prime$. Then there exists $z \in F$ such that $\|z-z^\prime\| < \de$. Thus $\|z^{\prime} - v\|\leq \|z^{\prime}-z\| + \|z-v\| < \de + \al < \be + 2\de.$
					Hence for each $z^\prime \in F^\prime$, $\|z^\prime - v\| <\be + 2\de$. Let $v_0 \in \emph{cent}_{V}(F^{\prime})$. Define $v^\prime = \la v + (1 - \la)v_0$ where $\la = \frac{\be - R_{F^\prime}}{\be - R_{F^\prime} + 2 \de}$. Then $v^\prime \in V$ and for each $z^\prime \in F^\prime$, 
					\begin{equation*}
						\begin{split}
							\|z^\prime - v^\prime\| & \leq \la \|z^\prime -v\| + (1-\la) \|z^\prime - v_0\| \\&< \la (\be + 2\de) + (1-\la) R_{F^\prime}= \la (\be - R_{F^\prime} + 2 \de) + R_{F^\prime} = \be.
						\end{split}
					\end{equation*}
					It follows that $v^\prime \in \bigcap_{z^\prime \in F^\prime} B[z^\prime, \be] \cap V$.
					Further, let $z^\prime \in F^\prime$. Since $d_{H}(F,F^\prime)< \de$, there exists $z \in F$ such that $\|z-z^\prime\|<\de$. Now,
					\begin{align}
						\begin{split}
							\|v-v^\prime\| &= (1-\la) \|v-v_0\|=\frac{2 \de}{\be - R_{F^\prime} + 2 \de} \|v-v_0\|\\&<\frac{2 \de}{\ga} ( \|v-z\| + \|z-z^\prime\| + \|z^\prime -v_0\|) \\&< \frac{2 \de}{\ga}(\al + \de + r(v_0,F^\prime)) =\frac{2\de}{\ga} (\al + \de + R_{F^\prime}) \\&< \frac{2 \de }{\ga} ( \al + R_{F} + 2 \de)<\frac{ 2 \de L}{\ga} \leq \e.
						\end{split}
					\end{align} 
					Therefore, $v^\prime \in \bigcap_{z^\prime \in F^\prime} B[z^\prime, \be] \cap V$ such that $\|v-v^\prime\| < \e$.
					
					Similarly, for each $w^\prime \in \bigcap_{z^\prime \in F^\prime} B[z^\prime, \be] \cap V$, we obtain $w \in \bigcap_{z \in F} B[z, \al] \cap V$ such that $\|w - w^\prime\|< \e$. This completes the proof.
				\end{proof}	
			
				The next result proves the stability of property-$(P_1)$ under $\ell_{\iy}$-direct sums.
				\begin{thm}\label{T4}
					For $i \in \{1,2\}$, let $X_i$ be a Banach space and $V_{i} \in \mc{CV}(X_i)$. Let $X = X_{1} \oplus_{\iy} X_2$ and $V = V_{1} \times V_2$. If for $i \in \{1,2\}$, $(X_i, V_i, \mc{CB}(X_{i}))$ has property-$(P_1)$ then $(X,V,\mc{CB}(X))$ has property-$(P_1)$.
				\end{thm}
				\begin{proof}
					Assume that for $i \in \{1,2\}$, $(X_i, V_i, \mc{CB}(X_{i}))$ has property-$(P_1)$. By Proposition~\ref{Prop1}, $(V,\mc{CB}(X))$ has \trcp. Let $B \in \mc{CB}(X)$ and $\e>0$. Then there exists $\de>0$ such that for $i \in \{1,2\}$,
					\begin{equation}\label{Eqn4}
						\emph{cent}_{V_{i}}(B(i), \de) \ci \emph{cent}_{V_{i}}(B(i)) + \e B_{X_i}.
					\end{equation}
					
					{\sc Case $1\colon$ }$r_{1}(B) = r_{2}(B)$
					
					By Remark~\ref{rem1} (i), $\emph{cent}_{V}(B) = \emph{cent}_{V_1}(B(1)) \times \emph{cent}_{V_2}(B(2))$. Similarly, for each $\ga>0$, we also have $\emph{cent}_{V}(B,\ga) = \emph{cent}_{V_1}(B(1),\ga) \times \emph{cent}_{V_2}(B(2),\ga)$. Therefore, it follows from $(\ref{Eqn4})$ that \[\emph{cent}_{V}(B,\de) \ci \emph{cent}_{V}(B) + \e B_{X}.\]
					
					{\sc Case $2\colon$ }$r_{1}(B) \neq r_{2}(B)$
					
					Without loss of generality, assume $r_{1}(B) < r_{2}(B)$, since the same arguments work for the reverse inequality. 
					By Remark~\ref{rem1} (i), $\emph{cent}_{V}(B) = \bigcap_{b_{1} \in B(1)} B_{X_1}[b_{1}, r_{2}(B)] \cap V_1 \times \emph{cent}_{V_{2}}(B(2))$. Similarly, for each $\ga>0$, \[\emph{cent}_{V}(B,\ga) = \bigcap_{b_{1} \in B(1)} B_{X_1}[b_{1}, r_{2}(B)+\ga] \cap V_1 \times  \emph{cent}_{V_2}(B(2),\ga).\]
					
					Let $\e>0$. Now replacing $X$, $V$ and $B$ by $X_1$, $V_1$ and $B(1)$ respectively in Lemma~\ref{L4.6.1}, we obtain $\de^{\prime}>0$ such that for each $A \in \mc{CB}(X)$ with $d_{H}(B,A)< 2 \de^{\prime}$ and a scalar $\be$ with $\vert \al -\be \vert <2 \de^{\prime}$, we have \[d_{H}\left(\bigcap_{b_1 \in B(1)} B_{X_1}[b_1, \al] \cap V_1,\bigcap_{a_1 \in A(1)} B_{X_1}[a_1, \be] \cap V_1\right) <\e.\]
					Choose $\de_0 = \min \{\de, \de^{\prime}\}$.
					Let us take $A= B$, $\al =r_{2}(B)$ and $\be = r_{2}(B) + \de_0$. Then we have 
					\begin{equation}\label{Eqn5}
						d_{H}\left(\bigcap_{b_1 \in B(1)} B_{X_1}[b_1, r_{2}(B)] \cap V_1,\bigcap_{b_1 \in B(1)} B_{X_1}[b_1, r_{2}(B) + \de_0] \cap V_1\right) <\e.
					\end{equation}
					
					Let $v=(v_1,v_2) \in \emph{cent}_{V}(B,\de_0)$ such that $v_1 \in \bigcap_{b_{1} \in B(1)} B_{X_1}[b_{1}, r_{2}(B)+\de_0] \cap V_1$ and $v_2 \in \emph{cent}_{V_2}(B(2),\de_0)$. Then by $(\ref{Eqn4})$ and $(\ref{Eqn5})$, there exists $w_1 \in \bigcap_{b_1 \in B(1)} B_{X_1}[b_1, r_{2}(B)] \cap V_1$ and $w_2 \in \emph{cent}_{V_2}(B(2))$ such that for each $i \in \{1,2\}$, $\|w_i - v_i\|\leq \e$. Hence $(w_1,w_2) \in \emph{cent}_{V}(B)$ and $\|v-(w_1,w_2)\| \leq \e$.
				\end{proof}
				
				One of the instances where the converses of Proposition~\ref{Prop3} and Theorem~\ref{T4} hold true is as follows:
				\begin{prop}
					For each $i \in \{1,2\}$, let $X_i$ be a Banach space and $Y_i$ be a closed non-trivial subspace of $X_i$. Let $X = X_{1} \oplus_{\iy} X_2$ and $Y = Y_{1} \oplus_{\iy} Y_2$. Then
					\blr
					\item If $(Y,\mc{CB}(X))$ has \trcp~then for each $i \in \{1,2\}$, $(Y_i,\mc{CB}(X_i))$ has \trcp.
					\item If $(X,Y,\mc{CB}(X))$ has property-$(P_1)$ then for each $i \in \{1,2\}$, $(X_i,Y_i,\mc{CB}(X_i))$ has property-$(P_1)$
					\el
				\end{prop}
				\begin{proof}
					{$(i)\colon$} Without loss of generality, we only prove that $(Y_1, \mc(CB)(X_1))$ has \trcp. Let Let $B_1 \in \mc{CB}(X_1)$. We can choose $B_2 \in \mc{CB}(X_2)$ such that $\emph{rad}_{Y_2}(B_2)=\emph{rad}_{Y_1}(B_1)$. Let $B = B_1 \times B_2 \in \mc{CB}(X)$. Then $B(1)=B_1$ and $B(2) = B_2$ and hence $r_{1}(B) = r_{2}(B)$. By Remark~\ref{rem1} (i), $\emph{cent}_{Y}(B) = \emph{cent}_{Y_1}(B_1) \times \emph{cent}_{Y_2}(B_2)$. Thus by our assumption, it follows that $\emph{cent}_{Y_1}(B_1) \neq \es$.
					 
					{$(ii)\colon$} Without loss of generality, we only prove that $(X_1, Y_1, \mc(CB)(X_1))$ has property-$(P_1)$. By Proposition~\ref{Prop1}, $(Y_1,\mc{CB}(X_1))$ has \trcp. Let $B_1 \in \mc{CB}(X_1)$. We choose $B_2 \in \mc{CB}(X_2)$ such that $\emph{rad}_{Y_2}(B_2)=\emph{rad}_{Y_1}(B_1)$. Define $B = B_1 \times B_2 \in \mc{CB}(X)$. Then $B(1)=B_1$ and $B(2) = B_2$ and hence $r_{1}(B) = r_{2}(B)$. Let $\e>0$. Then there exists $\de>0$ such that \[\emph{cent}_{Y}(B,\de) \ci \emph{cent}_{Y}(B) + \e B_{X}.\]
					By Remark~\ref{rem1} (i), $\emph{cent}_{Y}(B) = \emph{cent}_{Y_1}(B_1) \times \emph{cent}_{Y_2}(B_2)$. Similarly, $\emph{cent}_{Y}(B,\de) = \emph{cent}_{Y_1}(B_1,\de) \times \emph{cent}_{Y_2}(B_2, \de)$. Let $y_1 \in \emph{cent}_{Y_1}(B_1,\de)$ and $y_2 \in \emph{cent}_{Y_2}(B_2, \de)$. Hence $(y_1,y_2) \in \emph{cent}_{Y}(B,\de)$. Then there exists $y_1^{\prime} \in \emph{cent}_{Y_1}(B_1)$ and $y_{2}^{\prime} \in \emph{cent}_{Y_2}(B_2)$ such that $\|(y_1,y_2) - (y_1^{\prime},y_{2}^{\prime})\| \leq \e$. Thus $\|y_1 - y_{1}^{\prime}\|\leq \e$.
				\end{proof}
				
				An application of the above stability results is the following result.
				\begin{prop}\label{T5}
					Let $Y$ be a proximinal finite co-dimensional subspace of $c_0$. Then the triplet $(\ell_{\iy}, B_{Y}, \mc{CB}(\ell_{\iy}))$ has property-$(P_1)$.
				\end{prop}
				\begin{proof}
					Let $Y$ be a proximinal finite co-dimensional subspace of $c_0$. The following decomposition is well-known: the spaces $\ell_{\iy}$ and $Y$ are isometrically isomorphic to $X \oplus_{\iy} \ell_{\iy}$ and $Z \oplus_{\iy} c_0$ respectively for some finite dimensional spaces $X$ and $Z$ with $Z \ci X\ci c_0$  (see \cite[Section~4]{VI} for a detailed explanation). Clearly, $B_{Y} = B_{Z} \times B_{c_0}$. Using a compactness argument, we easily observe that $(X, B_{Z},\mc{CB}(X))$ has property-$(P_1)$. Therefore by Corollary~\ref{C1} and Theorem~\ref{T4}, we conclude that $(\ell_{\iy}, B_Y, \mc{CB}(\ell_{\iy}))$ has property-$(P_1)$. 
				\end{proof}

				\section{Semi-continuity of restricted Chebyshev-center maps of $\ell_{\iy}$-direct sums}\label{S4}
				In this section, we first derive few stability results concerning the semi-continuity properties of restricted Chebyshev-center maps of $\ell_{\iy}$-direct sums.
				
				\begin{prop}\label{T1}
					For each $i \in \{1,2\}$, let $X_i$ be a Banach space and $V_{i} \in \mc{CV}(X_i)$ such that $(V_i, \mc{CB}(X_{i}))$ has \trcp. Let $X = X_{1} \oplus_{\iy} X_2$ and $V = V_{1} \times V_2$. If for each $i \in \{1,2\}$, $\textrm{cent}_{V_{i}}(.)$ is \lc~on $\mc{CB}(X_{i})$ then the map $\textrm{cent}_{V}(.)$ is \lc~on $\mc{CB}(X)$. 
				\end{prop}
				\begin{proof}
					By Proposition~\ref{Prop1}, $(V,\mc{CB}(X))$ has \trcp. Let $B \in \mc{CB}(X)$ and $\e>0$. Using lower Hausdorff semi-continuity of the maps $\emph{cent}_{V_{1}}(.)$ and $\emph{cent}_{V_{2}}(.)$ at $B(1)$ and $B(2)$ respectively, there exists $\de>0$ such that for $i \in \{1,2\}$, if $A \in \mc{CB}(X)\mbox{ with }d_{H}(B,A) < \de\mbox{ and }\rho_{i} \in \emph{cent}_{V_{i}}(B(i))$ then
					\begin{equation}\label{Eqn1}
						B_{X_{i}}(\rho_{i},\e) \cap \emph{cent}_{V_{i}}(A(i)) \neq \es.
					\end{equation}
					
					{\sc Case $1\colon$ }$r_{1}(B) = r_{2}(B)$
					
					By Remark~\ref{rem1} (i), $\emph{cent}_{V}(B) = \emph{cent}_{V_1}(B(1)) \times \emph{cent}_{V_2}(B(2))$.
					Let $i \in \{1,2\}$. Let $\rho_{i} \in \emph{cent}_{V_{i}}(B(i))$ and $A \in \mc{CB}(X)$ such that $d_{H}(B,A) < \de$. Then by Remark~\ref{rem1} (ii), $d_{H}(B(i), A(i)) < \de$. Hence by $(\ref{Eqn1})$, there exists $v_{i} \in B_{X_{i}}(\rho_{i},\e) \cap \emph{cent}_{V_{i}}(A(i)).$ By Remark~\ref{rem1} (i), $(\rho_1,\rho_2) \in \emph{cent}_{V}(B)$ and $(v_1,v_2) \in \emph{cent}_{V}(A)$. Moreover, $\|(\rho_1,\rho_2) - (v_1,v_2)\| < \e.$ Hence $B_{X}((\rho_1,\rho_2),\e) \cap \emph{cent}_{V}(A)\neq \es$. Thus $\emph{cent}_{V}(.)$ is \lc~at $B$.
					
					{\sc Case $2\colon$ } $r_{1}(B) \neq r_{2}(B)$
					
					Without loss of generality, assume $r_{1}(B) < r_{2}(B)$, since the same arguments work for the reverse inequality. Let $2 \ga = r_{2}(B) - r_{1}(B)$. Replacing $X$, $V$, $B$ and $\al$ by $X_1$, $V_1$, $B(1)$ and $r_{2}(B)$ respectively in Lemma~\ref{L4.6.1}, we obtain $0<\de<\frac{\ga}{2}$ such that whenever $A \in \mc{CB}(X)$ with $d_{H}(B,A)< \de$, we have $r_{2}(A) - r_{1}(A) > \ga$ and 
					\begin{equation}\label{Eqn14}
						d_{H}\left(\bigcap_{b_1 \in B(1)} B_{X_1}[b_1, r_{2}(B)] \cap V_1,\bigcap_{a_1 \in A(1)} B_{X_1}[a_1, r_{2}(A)] \cap V_1\right) <\e.
					\end{equation}
					
					Without loss of generality, assume that $\de$ is so chosen that $(\ref{Eqn1})$ is also satisfied. Let $A \in \mc{CB}(X)$ with $d_{H}(B,A) < \de$. Then $r_{2}(A)> r_1(A)$ and hence by Remark~\ref{rem1} (i), $\emph{cent}_{V}(A) = \bigcap_{a_1 \in A(1)} B_{X_1}[a_1, r_{2}(A)] \cap V_1 \times \emph{cent}_{V_2}(A(2))$. Let $v_1 \in \bigcap_{b_1 \in B(1)} B_{X_1}[b_1, r_{2}(B)] \cap V_1$ and $v_2 \in \emph{cent}_{V_2}(B(2))$ and hence by Remark~\ref{rem1} (i), $v=(v_1,v_2) \in \emph{cent}_{V}(B)$. Therefore, by (\ref{Eqn1}) and (\ref{Eqn14}), there exists $w_1 \in \bigcap_{a_1 \in A(1)} B_{X_1}[a_1, r_{2}(A)] \cap V_1$ and $w_2 \in \emph{cent}_{V_2}(A(2))$ such that $\|v_1-w_1\| <\e$ and $\|v_2- w_2\|<\e$. Thus $ (w_1,w_2) \in B_{X}(v,\e) \cap \emph{cent}_{V}(A)$. 
				\end{proof}
				
				\begin{prop}\label{T2}
					For each $i \in \{1,2\}$, let $X_i$ be a Banach space and $V_{i} \in \mc{CV}(X_i)$. Let $X = X_{1} \oplus_{\iy} X_2$ and $V = V_{1} \times V_2$. If for each $i \in \{1,2\}$, $(X_i, V_i ,\mc{CB}(X_i))$ has property-$(P_1)$ then the map $\textrm{cent}_{V}(.)$ is \uc~on $\mc{CB}(X)$. 
				\end{prop}
				\begin{proof}
					By Proposition~\ref{Prop1}, $(V,\mc{CB}(X))$ has \trcp. By Theorem~\ref{Rem2}, for each $i \in \{1,2\}$, $\emph{cent}_{V_{i}}(.)$ is \uc~on $\mc{CB}(X_i)$. Let $B \in \mc{CB}(X)$ and $\e>0$. Then there exists $\de>0$ such that for each $i \in \{1,2\}$, if $A \in \mc{CB}(X)$ with
					
					\begin{equation}\label{Eqn11}
						d_{H}(B,A) < \de \Ra \emph{cent}_{V_i}(A(i)) \ci \emph{cent}_{V_i}(B(i)) + \e B_{X_{i}}.
					\end{equation}
					
					{\sc Case $1\colon$ }$r_{1}(B) = r_{2}(B)$
					
					By Remark~\ref{rem1} (i), $\emph{cent}_{V}(B) = \emph{cent}_{V_1}(B(1)) \times \emph{cent}_{V_2}(B(2))$. By our assumption and Remark~\ref{rem2} (ii), we can choose $\theta >0$ such that for $i \in \{1,2\}$, $\mc{S}(B(i),\theta)< \e$. We further choose $0 < \de< \frac{\theta}{4}$ such that $(\ref{Eqn11})$ holds valid. Let $A \in \mc{CB}(X)$ such that $d_{H}(B,A)< \de$. If $r_{1}(A) = r_{2}(A)$ then $\emph{cent}_{V}(A) = \emph{cent}_{V_1}(A(1)) \times \emph{cent}_{V_2}(A(2))$. Thus by $(\ref{Eqn11})$, we obtain $\emph{cent}_{V}(A) \ci \emph{cent}_{V}(B) + \e B_{X}.$ 
					
					Now without loss of generality, we assume that $r_{1}(A)< r_{2}(A)$. For $i \in \{1,2\}$, since \[\vert r_{i}(B) - r_{i}(A) \vert \leq d_{H}(B(i),A(i)) \leq d_{H}(B,A) < \frac{\theta}{4},\] we have 
					\begin{equation}\label{Eqn10}
						\vert r_{1}(B) - r_{2}(A) \vert \leq \vert r_{1}(B) - r_{2}(B) \vert + \vert r_{2}(B) - r_{2}(A) \vert < \frac{\theta}{4}.
					\end{equation}   
					Now, by Remark~\ref{rem1} (i), \[\emph{cent}_{V}(A) = \bigcap_{a_1 \in A(1)} B_{X_1}[a_1, r_{2}(A)] \cap V_1 \times \emph{cent}_{V_2}(A(2)).\] Let $v_1 \in \bigcap_{a_1 \in A(1)} B_{X_1}[a_1, r_{2}(A)] \cap V_1$. Then using $(\ref{Eqn10})$ and the assumption that $d_{H}(B,A)<\de$, for each $b_1 \in B(1)$, there exists $a_1 \in A(1)$ such that $\|a_1 - b_1\|< \frac{\theta}{4}$ and hence \[\|v_1 - b_1\| \leq \|v_1 - a_1\| + \|a_1 - b_1\| < r_2(A)+\frac{\theta}{4} < r_{1}(B) + \frac{\theta}{2}.\] It follows that $r(v_1, B(1)) \leq r_1(B) + \frac{\theta}{2}$. Now, since $\mc{S}(B(1),\theta)<\e$, $d(v_1,\emph{cent}_{V_1}(B(1)))<\e$. Thus there exists $w_1 \in \emph{cent}_{V_1}(B(1))$ satisfying $\|v_1-w_1\|<\e$. It follows that \[\bigcap_{a_1 \in A(1)}B_{X_1}[a_1, r_2(A)] \cap V_1 \ci \emph{cent}_{V_1}(B(1)) + \e B_{X_1}.\] Moreover, $\emph{cent}_{V_2}(A(2)) \ci \emph{cent}_{V_2}(B(2)) + \e B_{X_2}$. Hence, $\emph{cent}_{V}(A) \ci \emph{cent}_{V}(B) + \e B_{X}.$

					{\sc Case $2\colon$ }$r_{1}(B) \neq r_{2}(B)$ 	  
					
					We discuss only $r_{1}(B) < r_{2}(B)$ since the same arguments work for the reverse inequality. Let $2 \ga = r_{2}(B) - r_{1}(B)$. Replacing $X$, $V$, $B$ and $\al$ by $X_1$, $V_1$, $B(1)$ and $r_{2}(B)$ respectively in Lemma~\ref{L4.6.1}, we obtain $0<\de<\frac{\ga}{2}$ such that whenever $A \in \mc{CB}(X)$ with $d_{H}(B,A)< \de$, we have $r_{2}(A) - r_{1}(A) > \ga$ and
					\begin{equation}\label{Eqn12}
						d_{H}\left(\bigcap_{b_1 \in B(1)} B_{X_1}[b_1, r_{2}(B)] \cap V_1,\bigcap_{a_1 \in A(1)} B_{X_1}[a_1, r_{2}(A)] \cap V_1\right) <\e.
					\end{equation}
					We choose $\de>0$ such that $(\ref{Eqn11})$ is also satisfied. If $A \in \mc{CB}(X)$ such that $d_{H}(B,A)< \de$ then $r_2(A)>r_1(A)$ and hence by Remark~\ref{rem1} (i), \[\emph{cent}_{V}(A) = \bigcap_{a_1 \in A(1)} B_{X_1}[a_1, r_{2}(A)] \cap V_1\ \times \emph{cent}_{V_2}(A(2)).\] Let $A \in \mc{CB}(X)$ such that $d_{H}(B,A)<\de$. If $v_1 \in \bigcap_{a_1 \in A(1)} B_{X_1}[a_1, r_{2}(A)] \cap V_1$ and $v_2 \in \emph{cent}_{V_2}(A(2))$ then by $(\ref{Eqn11})$ and $(\ref{Eqn12})$, we choose $w_1 \in \bigcap_{b_1 \in B(1)} B_{X_1}[b_1, r_2(B)] \cap V_1$ and $w_2 \in \emph{cent}_{V_2}(B(2))$ satisfying $\|v_i-w_i\|< \e$ for $i\in  \{1,2\}$. Thus $(w_1,w_2) \in \emph{cent}_{V}(B)$ and $\|(v_1,v_2)-(w_1,w_2)\| < \e$. It follows that $\emph{cent}_{V}(A) \ci \emph{cent}_{V}(B) + \e B_{X}$.
				\end{proof}

				The following result follows from Propositions~\ref{T1} and $\ref{T2}$ and \cite[Remark~2.8]{VI}.
				\begin{prop}\label{T6}
					For each $i \in \{1,2\}$, let $X_i$ be a Banach space and $V_{i} \in \mc{CV}(X_i)$ such that $(X_i, V_i ,\mc{CB}(X_i))$ has property-$(P_1)$. Let $X = X_{1} \oplus_{\iy} X_2$ and $V = V_{1} \times V_2$. If for each $i\in \{1,2\}$, $\textrm{cent}_{V_i}(.)$ is Hausdorff metric continuous on $\mc{CB}(X_i)$ then $\textrm{cent}_{V}(.)$ is Hausdorff metric continuous on $\mc{CB}(X)$.
				\end{prop}
				
				We note here that the stability results proved in Sections~\ref{S3} and $\ref{S4}$ are similarly valid if we replace the class of all non-empty closed bounded subsets with the class of all non-empty compact or finite subsets of the respective spaces.
				
				We recall that a finite dimensional Banach space $X$ is called \textit{polyhedral} if $B_X$ has only finitely many extreme points. An infinite dimensional Banach space $X$ is called polyhedral if each of the finite dimensional subspace of $X$ is polyhedral. The space $c_0$ is a well-known example of an infinite dimensional polyhedral space. We refer to \cite{GM} and the references therein for a study on polyhedral spaces. I. G. Tsar\textquoteright kov recently proved
				\begin{thm}[{\cite[p.~243]{IGT} and \cite[Theorem~6.7, pg.~801]{AT}}]\label{T3}
					Let $V$ be a non-empty polyhedral subset of a finite dimensional polyhedral Banach space $X$. Then the map $\textrm{cent}_{V}(.)$ is globally Lipschitz Hausdorff metric continuous on $\mc{CB}(X)$ and admits a Lipschitz selection.
				\end{thm}
			    
				We are now equipped enough to prove the following result. 
				\begin{prop}\label{T7}
					Let $Y$ be a proximinal finite co-dimensional subspace of $c_0$. Then the map $\textrm{cent}_{B_Y}(.)$ is Hausdorff metric continuous on $\mc{CB}(\ell_{\iy})$.
				\end{prop}
				\begin{proof}
					Consider the decompositions of $\ell_{\iy}$ and $Y$ as $X\oplus_{\iy} \ell_{\iy} $ and $Z \oplus_{\iy} c_0$ respectively where $X$ and $Z$ are defined as in the proof of Proposition~\ref{T5}. Clearly, $B_{Y} = B_{Z} \times B_{c_0}$. Now, $X$ is a finite dimensional subspace of $c_0$ and hence is a polyhedral space. Thus by Theorem~\ref{T3}, $\emph{cent}_{B_{Z}}(.)$ is Hausdorff metric continuous on $\mc{CB}(X)$. Therefore, we conclude from Corollary~\ref{C1} and Theorems~\ref{T4} and $\ref{T6}$ that $\emph{cent}_{B_Y}(.)$ is Hausdorff metric continuous on $\mc{CB}(\ell_{\iy})$.
				\end{proof}	
				
				Lastly, we discuss the interrelationship between the semi-continuity properties of maps $\emph{cent}_{B_X}(.)$ and $\emph{cent}_{X}(.)$ in a Banach space $X$.
				\begin{prop}\label{prop2}
					Let $X$ be a Banach space and $Y$ be a subspace of $X$. Let the pair $(B_Y,\mc{CB}(X))$ have \trcp.
					\blr
					\item For each $\la>0$ and $B \in \mc{CB}(X)$, $\textrm{cent}_{\la B_{Y}}(.)$ is \lc~(resp. \uc) at $B$ if and only if $\textrm{cent}_{B_{Y}}(.)$ is \lc~(resp. \uc) at $\frac{1}{\la} B$.
					\item If $\textrm{cent}_{B_{Y}}(.)$ is \lc~(resp. \uc) on $\mc{CB}(X)$ then $\textrm{cent}_{Y}(.)$ is \lc~(resp. \uc) on $\mc{CB}(X)$.
					\item For each $\la>0$ and $B \in \mc{CB}(X)$, $\textrm{cent}_{\la B_{Y}}(.)$ is Hausdorff metric continuous at $B$ if and only if $\textrm{cent}_{B_{Y}}(.)$ is Hausdorff metric continuous at $\frac{1}{\la} B$.
					\item If $\textrm{cent}_{B_{Y}}(.)$ is Hausdorff metric continuous on $\mc{CB}(X)$ then $\textrm{cent}_{Y}(.)$ is Hausdorff metric continuous on $\mc{CB}(X)$.
					\el
				\end{prop}
				\begin{proof}
					By \cite[Proposition~2.2]{T1}, $(Y,\mc{CB}(X))$ has \trcp. 
					
					{\sc $(i)$} We first assume that $\emph{cent}_{\la B_Y}(.)$ is \lc~at $B$. Let us fix $\e>0$. Then there exists $\de>0$ such that if $A \in \mc{CB}(X)$ with
					\begin{equation} \label{Eqn6}
						d_{H}(B,A) < \de\mbox{, }y \in \emph{cent}_{\la B_{Y}}(B) \Ra B_{X}(y,\la \e) \cap \emph{cent}_{\la B_{Y}}(A) \neq \es.
					\end{equation}
					
					We now set $\ga = \frac{\de}{\la}$. Let $A \in \mc{CB}(X)$ such that $d_{H}(\frac{1}{\la}B, A)< \ga$ and $y \in \emph{cent}_{B_{Y}}(\frac{1}{\la}B)$. This implies $d_{H}(B, \la A)< \ga \la = \de$ and from \cite[Lemma~$2.1$]{T1}, $\la y \in \emph{cent}_{\la B_{Y}}(B)$. Therefore, by \eqref{Eqn6}, let $z \in B_{X}(\la y,\la \e) \cap \emph{cent}_{\la B_{Y}}(\la A)$. Thus by \cite[Lemma~$2.1$]{T1}, $\frac{z}{\la} \in \emph{cent}_{B_Y}(A)$. It follows that $\|\frac{z}{\la} - y\| < \e$ and hence $B_{X}(y,\e) \cap \emph{cent}_{B_Y}(A) \neq \es$.      
					
					Conversely, let $\e>0$. Then there exists $\de>0$ such that if $A \in \mc{CB}(X)$ with
					\begin{equation}\label{Eqn2}
						\begin{split}
							d_{H}\left(\frac{1}{\la}B,A\right) < \de\mbox{, }y \in \emph{cent}_{B_{Y}}\left(\frac{1}{\la}B\right)	
							\Ra B_{X}\left(y,\frac{\e}{\la}\right) \cap \emph{cent}_{B_{Y}}(A) \neq \es.
						\end{split}	
					\end{equation}
					
					We now set $\ga = \la \de$. Let $A \in \mc{CB}(X)$ such that $d_{H}(B, A)< \ga$ and $y \in \emph{cent}_{\la B_{Y}}(B)$. This implies $d_{H}(\frac{1}{\la} B, \frac{1}{\la} A)< \frac{\ga}{\la} = \de$ and from \cite[Lemma~$2.1$]{T1}, $\frac{y}{\la} \in \emph{cent}_{B_{Y}}(\frac{1}{\la} B)$. Therefore, by \eqref{Eqn2}, let $z \in B_{X}(\frac{y}{\la},\frac{\e}{\la}) \cap \emph{cent}_{B_{Y}}(\frac{1}{\la} A)$. Thus by \cite[Lemma~$2.1$]{T1}, $\la z \in \emph{cent}_{\la B_Y}(A)$. It follows that $\|\la z - y\| < \e$ and hence $B_{X}(y,\e) \cap \emph{cent}_{B_Y}(A) \neq \es$. 
					
					{\sc $(ii)$} Let $B \in \mc{CB}(X)$ and $\la > \sup_{b \in B} \|b\| + \emph{rad}_{Y}(B)$. From our assumption, $\emph{cent}_{B_{Y}}(.)$ is \lc~at $\frac{1}{\la} B$. Therefore from $(i)$, $\emph{cent}_{\la B_{Y}}(.)$ is \lc~at $B$. Let $\e>0$. Then there exists $\de>0$ such that if $A \in \mc{CB}(X)$ with
					\begin{align}\label{Eqn22}
						\begin{split}
							d_{H}(B,A) < \de\mbox{, }y \in \emph{cent}_{\la B_{Y}}(B)
							\Ra B_{X}(y,\e) \cap \emph{cent}_{\la B_{Y}}(A) \neq \es.
						\end{split}
					\end{align}
					Let us choose $0 < \ga <\min \{2\de, \la - (\sup_{b \in B} \|b\| + \emph{rad}_{Y}(B))\}$. Let $A \in \mc{CB}(X)$ such that $d_{H}(B, A)< \frac{\ga}{2}$ and $y \in \emph{cent}_{Y}(B)$. Now by Lemma~\ref{L2}, \[\emph{rad}_{Y}(A) + \sup_{a \in A} \|a\| \leq \emph{rad}_{Y}(B) + \frac{\ga}{2} +\frac{\ga}{2} + \sup_{b \in B} \|b\| < \la.\] Hence by \cite[Lemma~2.1~(iii)]{T1}, $\emph{cent}_{Y}(B) = \emph{cent}_{\la B_Y}(B)$ and $\emph{cent}_{Y}(A) = \emph{cent}_{\la B_Y}(A)$. Therefore, it follows from $(\ref{Eqn22})$ that \[B_{X}(y,\e) \cap \emph{cent}_{B_{Y}}(A) \neq \es.\]
					
					Similar arguments as above apply for the \uc~part of the results in $(i)$ and $(ii)$ and hence we omit it. The result in {\sc $(iii)$} follows from $(i)$ and \cite[Remark~2.8]{VI} and the result in {\sc $(iv)$} follows from $(ii)$ and \cite[Remark~2.8]{VI}.
				\end{proof}
				
				\section{A variation of transitivity problem for property-$(P_1)$}\label{sec5}
				In this section, we provide answers to the transitivity type problem for property-$(P_1)$, stated in Question~\ref{Q2}, for certain cases. The following result generalizes the transitivity property proved in \cite[Proposition~3.2]{CT}.
				
				\begin{prop}\label{T4.6.2}
					Let $X$ be a Banach space, $Y$ be an $M$-summand in $X$ and $Z $ be a subspace of $Y$. If $(Y,Z,\mc{CB}(Y))$ satisfies property-$(P_1)$ then $(X,Z,\mc{CB}(X))$ satisfies property-$(P_1)$. 	    
				\end{prop}
				\begin{proof}
					Let $X = Y \oplus_{\iy} W$ for some subspace $W$ of $X$. Let $B \in \mc{CB}(X)$. We define
					\begin{equation*}
						\begin{split}
							&B(1) = \{y \in Y\colon \mbox{there exists }x \in B\mbox{ and }w \in W\mbox{ such that } x= y+w\}\mbox{ and }\\& B(2) = \{w \in W\colon \mbox{there exists }x \in B\mbox{ and }y \in Y\mbox{ such that } x= y+w\}.
						\end{split}
					\end{equation*}
					Clearly, $B(1) \in \mc{CB}(Y)$ and $B(2) \in \mc{CB}(W)$. 
					
					We first prove the following claim$\colon$ \[\emph{rad}_{Z} (B) = \max \{\emph{rad}_{Z}(B(1)), \sup_{w \in B(2)} \|w\|\}.\] 
					Let $z \in Z$ and $y \in B(1)$. Then there exists $x \in B$ and $w \in W$ such that $x=y+w$. Thus $\|z-y\| \leq \max \{\|z-y\| , \|w\|\} = \|z -x\| \leq r(z,B).$
					It follows that for each $z \in Z$, $r(z,B(1)) \leq r(z,B)$. Hence $\emph{rad}_Z(B(1)) \leq \emph{rad}_Z(B).$
					Now, let $w \in B(2)$ and $z \in Z$. Then there exists $x \in B$ and $y \in Y$ such that $x=y+w$. Thus $\|w\| \leq \max \{\|y-z\|, \|w\|\} = \|x-z\| \leq r(z,B).$
					It follows that $\sup_{w \in B(2)} \|w\| \leq \emph{rad}_{Z}(B)$. Thus 
					\begin{equation*}
						\max \left\{\emph{rad}_{Z}(B(1)),\sup_{w \in B(2)} \|w\|\right\} \leq \emph{rad}_{Z} (B).
					\end{equation*}
					Now, let $\e>0$ be such that $\max \left\{\emph{rad}_{Z}(B(1)),\sup_{w \in B(2)} \|w\|\right\} < \e$. Let us choose $0<\de< \e$ such that $\emph{rad}_{Z}(B(1)) < \de$ and $\sup_{w \in B(2)} \|w\| < \de$. Thus $r(z_0, B(1)) < \de$ for some $z_0 \in Z$. Let $x= y+w \in B$ such that $y \in B(1)$ and $w \in B(2)$. Thus $\|x-z_0\| = \max \{\|y-z_0\|, \|w\|\} < \de.$ It follows that $r(z_0, B) \leq \de < \e$ and hence $\emph{rad}_{Z}(B) < \e$. Since $\e>0$ is arbitrary, it follows that \[\emph{rad}_{Z} (B) \leq \max \left\{\emph{rad}_{Z}(B(1)),\sup_{w \in B(2)} \|w\|\right\}.\]This proves our claim.  
					
					It is easy to verify that $\emph{cent}_{Z}(B(1)) \ci \emph{cent}_{Z}(B)$, using the claim above
					
					{\sc Case 1$\colon$} $\sup_{w \in B(2)} \|w\| > \emph{rad}_{Z}(B(1))$
					
					Let us define $R= \emph{rad}_{Z}(B) = \sup_{w \in B(2)} \|w\|$. It is easy to verify that $\emph{cent}_Z(B) = \bigcap_{y \in B(1)} B_{X}[y,R] \cap Z$ and for each $\eta>0$, $\emph{cent}_{Z}(B,\eta) = \bigcap_{y \in B(1)} B_{X}[y,R+\eta] \cap Z$. 
					
					Let us fix $\e>0$. Then by Lemma~\ref{L4.6.1}, there exists $\de>0$ such that for each $B^\prime \in \mc{CB}(X)$ with $d_{H}(B(1),B^\prime) < 2 \de$ and a scalar $\be>0$ with $|\be - R| < 2 \de$, we have \[d_{H}\left(\bigcap_{y \in B(1)} B_{X}[y,R] \cap Z,\bigcap_{b \in B^{\prime}} B_{X}[b,\be] \cap Z\right)<\e.\] Let us choose $B^{\prime} = B(1)$ and $\be= R+ \de$. Then  \[d_{H}\left(\bigcap_{y \in B(1)} B_{X}[y,R] \cap Z,\bigcap_{y \in B(1)} B_{X}[y,R+\de] \cap Z\right)<\e.\] Thus $\bigcap_{y \in B(1)} B_{X}[y,R+\de] \cap Z \ci \bigcap_{y \in B(1)} B_{X}[y,R] \cap Z + \e B_{X}$, that is, $\emph{cent}_Z(B,\de) \ci \emph{cent}_{Z}(B)  +\e B_X.$ This proves that $(X,Y,\{B\})$ satsifies property-$(P_1)$.
					
					{\sc Case 2$\colon$} $\sup_{w \in B(2)} \|w\| \leq \emph{rad}_{Z}(B(1))$
					
					Clearly, $\emph{rad}_{Z}(B) = \emph{rad}_{Z}(B(1))$. Let $z \in \emph{cent}_{Z}(B)$. Then $r(z,B) = \emph{rad}_{Z}(B(1))$. Let $y \in B(1)$. Then there exists $x \in B$ and $w \in W$ such that $x =y+w$. Hence $\|z-y\| \leq \|z-x\| \leq r(z,B) =\emph{rad}_{Z}(B(1)).$ It follows that $z \in \emph{cent}_{Z}(B(1))$. Therefore, $\emph{cent}_{Z}(B)= \emph{cent}_{Z}(B(1))$. Similarly, for each $\eta>0$, $\emph{cent}_{Z}(B,\eta)= \emph{cent}_{Z}(B(1),\eta)$. Since $(Y,Z,\{B(1)\})$ satisfies property-$(P_1)$, it follows that $(X,Y,\{B\})$ has property-$(P_1)$.
				\end{proof}	
				
				Another instance where Question~\ref{Q2} is positively answered is as follows$\colon$
				\begin{prop}\label{P4.6.5}
					Let $X$ be an $L_1$-predual space. Let $Y$ be a finite co-dimensional subspace of $X$ and $J$ be an $M$-ideal in $X$ such that $Y \ci J$. If $Y$ is strongly proximinal in $J$, then the triplet $(X,Y,\mc{K}(X))$ satisfies property-$(P_1)$.
				\end{prop}
				\begin{proof}
					By Theorem~\ref{T9}, $Y$ is strongly proximinal in $X$. Therefore, by \cite[Theorem~4.7]{T1}, $(X,Y,\mc{K}(X))$ satisfies property-$(P_1)$.
				\end{proof}
			
				\section*{Declarations}
				\begin{enumerate}
					\item {\sc Funding:} No funding was received to assist with the preparation of this manuscript.
					\item {\sc Competing interests:} The author has no competing interests to declare that are relevant to the content of this manuscript.
				\end{enumerate}
				
				 \bibliographystyle{plain}
				\bibliography{sofc}

\begin{thebibliography}{10}

\bibitem{AT}
A.~R. Alimov and I.~G. Tsar'kov.
\newblock The {C}hebyshev center of a set, the {J}ung constant, and their
  applications.
\newblock {\em Uspekhi Mat. Nauk}, 74(5(449)):3--82, 2019.
\newblock DOI: \href{https://doi.org/10.4213/rm9839}{10.4213/rm9839}.

\bibitem{Amir}
Dan Amir.
\newblock Chebyshev centers and uniform convexity.
\newblock {\em Pacific J. Math.}, 77(1):1--6, 1978.

\bibitem{DuNa2}
S.~Dutta and Darapaneni Narayana.
\newblock Strongly proximinal subspaces in {B}anach spaces.
\newblock In {\em Function spaces}, volume 435 of {\em Contemp. Math.}, pages
  143--152. Amer. Math. Soc., Providence, RI, 2007.
\newblock DOI:
  \href{https://doi.org/10.1090/conm/435/08372}{10.1090/conm/435/08372}.

\bibitem{GM}
Alan Gleit and Robert McGuigan.
\newblock A note on polyhedral {B}anach spaces.
\newblock {\em Proc. Amer. Math. Soc.}, 33:398--404, 1972.
\newblock DOI: \href{https://doi.org/10.2307/2038069}{10.2307/2038069}.

\bibitem{GV}
G.~Godefroy and V.~Indumathi.
\newblock Proximinality in subspaces of {$c_0$}.
\newblock {\em J. Approx. Theory}, 101(2):175--181, 1999.
\newblock DOI:
  \href{https://doi.org/10.1006/jath.1999.3382}{10.1006/jath.1999.3382}.

\bibitem{GI}
G.~Godefroy and V.~Indumathi.
\newblock Strong proximinality and polyhedral spaces.
\newblock {\em Rev. Mat. Complut.}, 14(1):105--125, 2001.
\newblock DOI:
  \href{https://doi.org/10.5209/rev_REMA.2001.v14.n1.17047}{10.5209/rev\_REMA.2001.v14.n1.17047}.

\bibitem{HWW}
P.~Harmand, D.~Werner, and W.~Werner.
\newblock {\em {$M$}-ideals in {B}anach spaces and {B}anach algebras}, volume
  1547 of {\em Lecture Notes in Mathematics}.
\newblock Springer-Verlag, Berlin, 1993.
\newblock DOI: \href{https://doi.org/10.1007/BFb0084355}{10.1007/BFb0084355}.

\bibitem{VI}
V.~Indumathi.
\newblock Semi-continuity of metric projections in {$l_\infty$}-direct sums.
\newblock {\em Proc. Amer. Math. Soc.}, 133(5):1441--1449, 2005.
\newblock DOI:
  \href{https://doi.org/10.1090/S0002-9939-04-07690-7}{10.1090/S0002-9939-04-07690-7}.

\bibitem{CRLa}
C.~R. Jayanarayanan and S.~Lalithambigai.
\newblock Strong ball proximinality and continuity of metric projection in
  {$L_1$}-predual spaces.
\newblock {\em J. Approx. Theory}, 213:120--135, 2017.
\newblock DOI:
  \href{https://doi.org/10.1016/j.jat.2016.08.003}{10.1016/j.jat.2016.08.003}.

\bibitem{CT}
C.~R. Jayanarayanan and Tanmoy Paul.
\newblock Strong proximinality and intersection properties of balls in {B}anach
  spaces.
\newblock {\em J. Math. Anal. Appl.}, 426(2):1217--1231, 2015.
\newblock DOI:
  \href{https://doi.org/10.1016/j.jmaa.2015.01.013}{10.1016/j.jmaa.2015.01.013}.

\bibitem{Lacey}
H.~Elton Lacey.
\newblock {\em The isometric theory of classical {B}anach spaces}.
\newblock Die Grundlehren der mathematischen Wissenschaften, Band 208.
  Springer-Verlag, New York-Heidelberg, 1974.

\bibitem{LPST}
S.~Lalithambigai, T.~Paul, P.~Shunmugaraj, and V.~Thota.
\newblock Chebyshev centers and some geometric properties of {B}anach spaces.
\newblock {\em J. Math. Anal. Appl.}, 449(1):926--938, 2017.
\newblock DOI:
  \href{https://doi.org/10.1016/j.jmaa.2016.12.026}{10.1016/j.jmaa.2016.12.026}.

\bibitem{LL}
A.~J. Lazar and J.~Lindenstrauss.
\newblock Banach spaces whose duals are {$L_{1}$} spaces and their representing
  matrices.
\newblock {\em Acta Math.}, 126:165--193, 1971.
\newblock DOI: \href{https://doi.org/10.1007/BF02392030}{10.1007/BF02392030}.

\bibitem{Mach2}
Jaroslav Mach.
\newblock Continuity properties of {C}hebyshev centers.
\newblock {\em J. Approx. Theory}, 29(3):223--230, 1980.
\newblock DOI:
  \href{https://doi.org/10.1016/0021-9045(80)90127-6}{10.1016/0021-9045(80)90127-6}.

\bibitem{PN}
D.~V. Pai and P.~T. Nowroji.
\newblock On restricted centers of sets.
\newblock {\em J. Approx. Theory}, 66(2):170--189, 1991.
\newblock DOI:
  \href{https://doi.org/10.1016/0021-9045(91)90119-U}{10.1016/0021-9045(91)90119-U}.

\bibitem{Seever}
G.~L. Seever.
\newblock Nonnegative projections on {$C_{0}(X)$}.
\newblock {\em Pacific J. Math.}, 17:159--166, 1966.

\bibitem{T3}
T.~Thomas.
\newblock On {K}akutani's characterization of the closed linear sublattices of
  {$C(X)$} -- {R}evisited.
\newblock ({P}reprint available at \url{https://arxiv.org/abs/2209.04140}),
  2022.

\bibitem{T1}
T.~Thomas.
\newblock On {P}roperty-{$(P_1)$} in {B}anach {S}paces.
\newblock {\em Journal of Convex Analysis}, 29(4):975--994, 2022.

\bibitem{IGT}
I.~G. Tsar{\textquoteright}kov.
\newblock Stability of the relative {C}hebyshev projection in polyhedral
  spaces.
\newblock {\em Tr. Inst. Mat. Mekh.}, 24(4):235--245, 2018.
\newblock DOI:
  \href{https://doi.org/10.21538/0134-4889-2018-24-4-235-245}{10.21538/0134-4889-2018-24-4-235-245}.

\bibitem{Ves}
Libor Vesel\'{y}.
\newblock A {B}anach space in which all compact sets, but not all bounded sets,
  admit {C}hebyshev centers.
\newblock {\em Arch. Math. (Basel)}, 79(6):499--506, 2002.
\newblock DOI: \href{https://doi.org/10.1007/PL00012477}{10.1007/PL00012477}.

\end{thebibliography}
			\end{document}